\documentclass[11pt]{amsart}
\usepackage{amsmath}
\usepackage{amssymb}
\usepackage{amscd}
\usepackage{hyperref}
\usepackage{bbold}
\newtheorem{thm}{Theorem}
\newtheorem{cor}{Corollary}

\newtheorem{rem}{Remark}
\newtheorem{lem}{Lemma}
\newtheorem{df}{Definition}

\renewcommand{\Re}{\operatorname{Re}}
\renewcommand{\Im}{\operatorname{Im}}

\newcommand{\sll}{\mathfrak{sl}}

\newcommand{\C}{\mathbb C}

\newcommand{\R}{\mathbb R}
\newcommand{\tG}{\tilde G}
\newcommand{\tP}{\tilde P}

\newcommand{\al}{\alpha}

\newcommand{\G}{\mathcal G}
\newcommand{\om}{\omega}
\renewcommand{\th}{\theta}
\newcommand{\Om}{\Omega}
\newcommand{\pp}{\mathfrak{p}}
\newcommand{\tom}{\tilde\om}

\newcommand{\g}{\mathfrak{g}}

\newcommand{\Mat}{\operatorname{Mat}}
\newcommand{\Hom}{\operatorname{Hom}}
\newcommand{\Ad}{\operatorname{Ad}}
\newcommand{\tr}{\operatorname{tr}}
\newcommand{\id}{\operatorname{id}}

\newcommand{\Z}{{\mathbb Z}}

\renewcommand{\gg}{{\mathfrak g}}

\newcommand{\I}{\operatorname{i}}
\newcommand{\Aut}{\operatorname{Aut}}

\newcommand{\GL}{\operatorname{GL}}
\newcommand{\PSU}{\operatorname{PSU}}
\newcommand{\SU}{\operatorname{SU}}
\newcommand{\gl}{\operatorname{{\mathfrak g\mathfrak l}}\,}
\newcommand{\su}{\operatorname{{\mathfrak s\mathfrak u}}\,}

\newcommand{\be}{{\beta}}

\def\sideremark#1{\ifvmode\leavevmode\fi\vadjust{
\vbox to0pt{\hbox to 0pt{\hskip\hsize\hskip1em
\vbox{\hsize3cm\tiny\raggedright\pretolerance10000
\noindent #1\hfill}\hss}\vbox to8pt{\vfil}\vss}}}
\begin{document}
\title{Free CR distributions}
\author{Gerd Schmalz and Jan Slov\'ak}
\date{}

\begin{abstract}
There are only some exceptional CR dimensions and codimensions such
that the geometries enjoy a discrete classification of the
pointwise types of the homogeneous models. The cases of CR
dimensions $n$ and codimensions $n^2$ are among the very few
possibilities of the so called parabolic geometries. Indeed, the
homogeneous model turns out to be $\PSU(n+1,n)/P$ with a suitable
parabolic subgroup $P$. We study the geometric properties of such
real $(2n+n^2)$-dimensional submanifolds in $\mathbb C^{n+n^2}$ for all
$n>1$. In
particular we show that the fundamental invariant is
of torsion type, we
provide its explicit computation, and we discuss an analogy to the
Fefferman construction of a circle bundle in the hypersurface type CR
geometry. 
\end{abstract}

\maketitle

\section{Introduction}

There is a vast amount of literature on analytical and geometrical aspects
of real submanifolds of complex spaces $\C^N$. The generic hypersurfaces in
$\C^{n+1}$, i.e. real $(2n+1)$-dimensional contact manifolds equipped with
complex structure on the contact distribution, represent the best known
example studied in detail for more than hundred years already. From the
geometrical point of view, the main reason for their nice and rich
structural behaviour lies in the algebraic properties of the Klein's
homogeneous model which is represented by the quadric $Q = \operatorname{PSU}(p+1,q+1)/P$
obtained from the standard action of $\operatorname{PSU}(p+1,q+1)$ on $\C^{p+q+2}$. The
space $Q$ itself coincides with the space of isotropic lines with respect to
the Hermitian form $h$ of signature $(p,q)$ and $P$ is the isotropic
subgroup of one such line.

The general case of CR geometries of CR dimension $n$ and codimension
$k$, i.e. real $2n+k$ surfaces in $\C^{2n+k}$ does not permit a similar
approach in general, but there are some exceptional dimensions and
codimensions which are very similar to the hypersurface case. These exceptional dimensions are $k=n^2$, $k=n^2-1$ for arbitrary $n>1$, $n=k=2$, $n=3, k=2$ and $n=3, k=7$.  The case $n=k=2$ was studied in \cite{SS1} and provides a beautiful way of viewing
real $6$-dimensional surfaces in $\C^4$. 
Nowadays, there is the general theory of parabolic geometries and the
originally surprising first example of a CR
geometry with a parabolic isotropy group in the semi-simple structure group
outside of the hypersurface type CR structures  provides a
quite easy example of
its applications, see Section 4.3 of \cite{CS09}.

In this paper we study another coincidence when the bracket generating
distribution inherited on a generic CR submanifold $M\subset \C^N$ allows
for only one type infinitesimal homogeneous model. The reason for this
exceptional behaviour is similar to the so called free $n$-dimensional
distributions studied intensively in geometric literature, cf. \cite{DS09}.
There the non-degeneracy of rank $n$ distribution in a space, where the
codimension is equal to the dimension of the space of all skew-symmetric
matrices, automatically leads to isomorphic Lie algebra structures on
associated graded tangent spaces at all points. In our case, the Levi form
is an imaginary part of a Hermitian form valued in the space of
skew-Hermitian matrices. Thus the codimension $n^2$ again coincides with the
dimension of the whole target space.
The paper \cite{DS09} has been also the main inspiration for most of the
algebraic technicalities here. This concerns in particular the Fefferman like
construction of a circle bundle equipped with a Hermitian analog of the
spinorial geometry for all
generic CR dimension $n$ and codimension $n^2$ geometries, in full analogy
to the hypersurface case.

Our approach is based on the recent theory of parabolic geometries, as
developed in \cite{CS09}, building itself on the Cartan-Tanaka theory, cf. 
\cite{tanaka, yamaguchi}.  This turns the quite deep problems on high
codimensional CR manifolds into rather straightforward applications of the
general methods and results (and opens new questions at the same time).  In
particular, the standard technique of the exterior differential systems
supported by the general results provides a very explicit and efficient
approach to the basic invariants.\medskip

{\bf Acknowledgements.} The research reflected in this paper has been supported
by the Czech Science Foundation, grant Nr. 201/08/0397. The first named author also gratefully acknowledges the support by Max-Planck-Institut f\"ur Mathematik in Bonn.  The authors
benefited from numerous discussions on the topic to M.G. Eastwood, B.
Doubrov, and A. Cap. The authors also wish to thank the referees for their comments that helped to improve the manuscript.

\section{The homogeneous model and Cartan connections}

Consider the $(2n+1)$-dimensional
complex space $\C^{2n+1}$ with coordinates
$(\epsilon_1,\dots,\epsilon_n,\zeta,\omega_1,\dots,\omega_n)$
endowed with the Hermitian form
$$
h=|\zeta|^2 + \sum_{\nu=1}^n \epsilon_\nu \bar{\omega}_\nu +
\omega_\nu \bar{\epsilon}_\nu
$$
of signature $(n+1,n)$. We shall write
$$
\Bbb J = \begin{pmatrix} 0 & 0 & I \\ 0&1&0 \\ I&0&0 \end{pmatrix}
$$
for the block matrix of this form in the standard coordinates on $\Bbb
C^{2n+1}$
(here $I=I_n$ is the unit matrix of rank $n$).

\subsection{The homogeneous quadric $Q$}

Let us consider the Grassmannian of $n$-dimensional complex subspaces of
$\C^{2n+1}$ and denote by $Q$ its subset consisting of the
isotropic subspaces with respect to $h$. By $\SU(n+1,n)$ we denote the
special pseudo-unitary group with respect to the Hermitian form
introduced above.

\begin{lem}\label{lem1}
$Q$ is a homogeneous CR-manifold of CR-dimension $n$ and
CR-codimension $n^2$ with rational transitive action of
$\SU(n+1,n)$. The kernel of the action is $\Z_{2n+1}$ and so the
effective homogeneous model is $Q=G/P$, where $G=\PSU(n+1,n)=
\SU(n+1,n)/\Bbb Z_{2n+1}$ and $P$ is the
isotropic subgroup of one fixed isotropic plane $V_0$ in $Q$.
\end{lem}

\begin{proof}
Obviously, the standard action of $\SU(n+1,n)$ on $\C^{2n+1}$
induces an action on $Q$. We show that it is transitive. In order to do this
we construct a pseudo-unitary basis of $\C^{2n+1}$ adapted to a chosen fixed
plane $V\in Q$. Let $v_1,\dots, v_n$ be a basis of $V$. Then the
$n \times (2n+1)$ matrix
$(v_1,\dots, v_n)^* \Bbb J$
has rank $n$ and so we can find $n$
vectors $(w_1,\dots,w_n)$ such that
$$
(v_1,\dots, v_n)^* \Bbb J (w_1,\dots,w_n)= I_{2n+1}
.$$

The vectors $v_1,\dots, v_n,w_1,\dots,w_n$ are linearly
independent. Indeed, if
$$
\lambda_1 v_1+\cdots + \lambda_n v_n+\mu_1 w_1+\cdots+\mu_n w_n=0
$$
then multiplication of the left
hand side with $v_i$ yields $\mu_i=0$ and the $v_i$
were linearly independent by assumption. Next,
let us write $A=(A_{ij})$ for the Hermitian
Gram matrix of the basis $(w_i)$, that is $A_{ij}=w_i^*\mathbb J w_j$. If we replace the
basis $(w_i)$ by $(w_i- \frac12A(v_i))$, then
$$
h( v_i, v_j) =0, \quad h(v_i, w_j) =\delta_{i,j}, \quad h(w_i, w_j)
=0.
$$
Finally choose $w_{n+1}$ orthogonal to $v_1,\dots,
v_n,w_1,\dots,w_n$. According to the signature of the inner
product, it has positive length. By scaling we achieve that the
length is $1$.

For any isotropic planes $V$ and $V'$ we
can now find such bases and the cooresponding coordinate change is pseudo-unitary with respect to
$h$. On the other hand, every such basis corresponds to an isotropic subspace of
dimension $n$. Therefore, $Q$ can be viewed as the orbit of one isotropic space $V\in Q$
with respect to the holomorphic action of $\SU(n+1,n)$.
Since $Q$ is a real submanifold of the complex Grassmannian manifold $\SU(n+1,n)$ acts transitively on $Q$ by holomorphic automorphisms of the ambient Grassmannian it must be a CR-manifold.

In order to see the CR structure on $Q$ in detail, we shall pass to a local
chart of the Grassmannian. Let $V_0$ be spanned by the standard vectors
$e_\nu$, for $\nu=1,\dots,n$. In a neighbourhood of $V_0$ in the Grassmanian
of $n$-planes we
introduce coordinates $z_{\nu}, w_{\mu\nu}$ such that a subspace $V$ is
given as the span of
$$
\begin{pmatrix}1\\0\\\vdots\\0\\z_1\\w_{11}\\\vdots\\w_{1n}\\
\end{pmatrix},
\begin{pmatrix} 0\\1\\\vdots\\0\\z_2\\w_{21}\\\vdots\\w_{2n}\\
\end{pmatrix},\dots,
\begin{pmatrix}0\\0\\\vdots\\1\\z_n\\w_{n1}\\\vdots\\w_{nn}\\ \end{pmatrix}
= \begin{pmatrix} I\\ z \\ W\end{pmatrix},
$$
i.e. $V_0$ is given as $(I,0,0)^T$. An $n$-plane $V=(I,z,W)^T$ of this chart is isotropic if
\[(I,z^*,W^*)\mathbb J (I,z,W)^T=0\]
that is
\begin{equation}\label{charteq}
W+W^*+z^*z=0.
\end{equation}
Notice that (\ref{charteq}) is the defining equation of $Q$ in the chart.

$V$ is the image of $V_0$ under the linear mapping
\[\begin{pmatrix} I & 0 & 0\\ z &1 & 0 \\ W & z^* & I \end{pmatrix}\colon \mathbb C^{2n+1}\to \mathbb C^{2n+1}.\]
This mapping is in $\SU(n+1,n)$ and it is uniquely determined it we imposed the conditions that it keeps the first $n$ coordinates unchanged and the $n+1$-st coordinate unchanged modulo the first $n$ coordinates. 

Thus, using the following block structure for the matrices in the Lie group $G=\SU(n+1,n)$ or its Lie
algebra $\gg=\su(n+1,n)$, 
\begin{align*}
\begin{matrix}\hspace{9mm} n \hspace{9mm}&
1 & \hspace{9mm} n \hspace{9mm} \end{matrix}&\\ \begin{pmatrix}
\fbox{\rule[-3mm]{0cm}{10mm} \hspace{3mm} 0 \hspace{3mm}} &
\fbox{\rule[-3mm]{0cm}{10mm} 1} & \fbox{\rule[-3mm]{0cm}{10mm} \hspace{3mm}
2 \hspace{3mm}}\\ \rule{0cm}{.1mm}& \rule{0cm}{.1mm}&\rule{0cm}{.1mm}\\
 \fbox{ \hspace{2.3mm} -1
\hspace{2.3mm}} & \fbox{ 0} & \fbox{ \hspace{3mm} 1
\hspace{3mm}}\\
\rule{0cm}{.1mm}& \rule{0cm}{.1mm}&\rule{0cm}{.1mm}\\
\fbox{\rule[-3mm]{0cm}{10mm} \hspace{2.5mm} -2 \hspace{2.5mm}} &
\fbox{\rule[-3mm]{0cm}{10mm}  -1} & \fbox{\rule[-3mm]{0cm}{10mm}
\hspace{3mm} 0 \hspace{3mm}}
\end{pmatrix}& \quad
\begin{matrix}\rule[-3mm]{0cm}{10mm}  n\\\rule{0cm}{.1mm}\\
1\\\rule{0cm}{.1mm}\\\rule[-3mm]{0cm}{10mm}  n
\end{matrix}
\end{align*}
we may identify the latter chart with the exponential image of the lower
block triangular matrices and the action of the group $\SU(n+1,n)$ on $Q$ by CR-automorphisms is just the adjoint
action of $\SU(n+1,n)$.

In particular, the isotropy subgroup of the origin $V_0 = (I,0,0)^T$ is the parabolic
subgroup $P$ of all block upper triangular matrices in $G=\SU(n+1,n)$. In summary, we
have identified the CR-manifold $Q$ with the compact partial flag variety $Q=G/P$.

The Lie algebra $\gg=\su(n+1,n)$ of $\SU(n+1,n)$ enjoys the natural grading
$$
\gg=\gg_{-2}\oplus \gg_{-1}\oplus \gg_{0}\oplus \gg_{1}\oplus
\gg_{2}
$$
given by the block structure shown above. A simple computation reveals that $\gg_{-1}=\C^n$,
$\gg_{-2}$ consists of all skew-Hermitian matrices and the Lie algebra bracket is given
by $[X,Y]=X^*Y - Y^*X\in \mathfrak g_{-2}$ for $X,Y\in \mathfrak g_{-1}$. 

Then the block diagonal matrices in $\gg_0$ are of the form $(-C,
2i\operatorname{Im}\operatorname{Tr}C, C^*)$ with $C\in \gl(n,\Bbb C)$ arbitrary,
$\gg_1=\C^n$, and $\gg_2$ consists again of skew-Hermitian matrices.

Now, the natural complex structure on the Grassmanian of $n$-planes keeps
the subspace $\gg_{-1}$ invariant, turning it into the CR-distribution on $Q$
at the origin, while the $\gg_{-2}$ provides the remaining $n^2$ codimensions.
The holomorphic transitive action of $\SU(n+1,n)$ extends this CR structure
over the entire $Q$.

Finally, the elements in the reductive part $G_0\subset P$ are the
block diagonal matrices $(C^{-1}, \al, C^*)$ with $\al = (\operatorname{det}
C)^2|\operatorname{det}C|^{-2}$. Its action on $V_0$ reads (expressed in the
analogy to projective coordinates $(I:Z:W)^T$)
$$
\begin{pmatrix}
I\\Z\\W\end{pmatrix}
\mapsto \begin{pmatrix} C^{-1}& 0 & 0 \\
0 & \al & 0 \\ 0 & 0 & C^*\end{pmatrix}
\begin{pmatrix} I \\ Z \\ W \end{pmatrix}
 = \begin{pmatrix} C^{-1}\\ \al Z \\ C^* W\end{pmatrix} \simeq
\begin{pmatrix} I \\ \al Z C \\ C^* W C\end{pmatrix}
.$$
Thus the kernel of the action of $G$ on $Q$ consists of matrices $C = \be I$
with $|\be|=1$ and $\al=\be^{2n}$. The conclusion is $\be^{2n+1}=1$ and this
proves the last claim.
\end{proof}

The group $\SU(n+1,n)$ acts holomorphically on the Grassmannian and
preserves $Q$ and so it is a subgroup of $\Aut Q$. 

The expression of the general action of $\operatorname{exp}\g_1$ and
$\operatorname{exp}\g_2$ in the coordinate patch introduced above is given
by the formulae (for small $Y$ and $T$)
\begin{align*}
\begin{pmatrix}
I\\Z\\W\end{pmatrix}
\mapsto &\begin{pmatrix} I& Y & -\frac12Y^*Y \\
0 & 1 & -Y^* \\ 0 & 0 & I\end{pmatrix}
\begin{pmatrix} I\\ Z \\ W\end{pmatrix}
= \begin{pmatrix} I+YZ - \frac12Y^*YW \\ Z-Y^*W \\ W \end{pmatrix}
\\&\simeq
\begin{pmatrix} I \\ (Z-Y^*W)(I+YZ - \frac12Y^*YW)^{-1}
\\ W(I+YZ - \frac12Y^*YW)^{-1}\end{pmatrix}
\\
\begin{pmatrix}I \\ Z \\ W\end{pmatrix}
\mapsto &\begin{pmatrix} I& 0 & T \\
0 & 1 & 0 \\ 0 & 0 & I\end{pmatrix}
\begin{pmatrix} I\\ Z \\ W\end{pmatrix}
= \begin{pmatrix} I+TW \\ Z \\ W \end{pmatrix}
\simeq
\begin{pmatrix} I \\ Z(I+TW)^{-1}
\\ W(I+TW)^{-1}\end{pmatrix}
.\end{align*} The equivalent formulae for general automorphisms have been computed also
directly by the standard methods of complex analysis (see, e.g., \cite{ES94}). 
But we shall see that actually all automorphisms of $Q$
are of this form as a simple consequence of our general theory below.

Clearly, the case $n=1$ recovers the real $3$-dimensional hypersurfaces in
$\Bbb C^2$. The algebraic properties of the model are quite different for
this lowest dimensional case and we
shall treat only the other cases $n>1$ in the sequel.

\subsection{General facts on parabolic geometries}
Let us briefly remind some general concepts, the reader can find all details
in the monograph \cite{CS09}. 

The parabolic geometries can be
viewed as curved deformations of the homogeneous spaces $G/P$ with $G$
semisimple and $P$ parabolic:

\begin{df}\label{df1}
A {\em
Cartan geometry of type $(G,P)$ on a manifold $M$} is a principal fiber
bundle $\mathcal G\to M$ with structure group $P$, equipped with an absolute
parallelism $\om\in\Om^1(\mathcal G,\mathfrak g)$ which is $\Ad$-invariant
with respect to the principal $P$-action and reproduces the fundamental
vector fields. The form $\om$ is called the {\em Cartan connection} of type
$(G,P)$ on $M$.
\end{df}

Most general features of the geometry in question are read off from the algebraic
properties of the {\em flat model} $G\to G/P$, where $\om$ is the
Maurer-Cartan form. The name {\em parabolic geometry} refers to cases where
$G$ is semisimple and $P\subset G$ parabolic. 
At the level of the curved geometries, the $P$-invariant filtration
inherited on $T\G$ from the absolute parallelism projects to 
the filtration on
$TM$. Let us also notice that the Cartan-Killing form identifies $\pp_+ =
\gg_1\oplus\dots\oplus \gg_k$ with
$(\gg/\pp)^*$ as $P$-modules and $\gg/\pp$ equals
to $\gg_-=\gg_{-k}\oplus\dots\oplus\gg_{-1}$ as $
G_0$-module, where $G_0$ is
the reductive part of the parabolic subgroup $P$ with Lie algebra $\g_0$.

Roughly speaking, the entire Cartan connection can be mostly recovered from
this filtration on the manifold $M$ using suitable normalization conditions. 
Thus, a parabolic geometry on a manifold $M$ is given by a
particular geometric structure visible at the manifold itself, while $\G$
and $\om$ are uniquely determined by a functorial construction.

The structures in question are the so
called {\em regular infinitesimal flag structures of type $(G,P)$} and they 
are derived from the grading
$
\gg=\gg_{-k}\oplus \dots \oplus \gg_k
$
of the semisimple Lie algebra $\gg$
giving rise to the parabolic subalgebra $\pp = \gg_0\oplus\dots\oplus\gg_k$.
In fact, the analogues of the $P$--invariant filtration on 
$\gg_{-k}\oplus \dots \oplus \gg_{-1}$  have to be given in the individual
tangent spaces.

In our case, such structure is given by a 
non-degenerate distribution of
the right dimension and codimension, satisfying some additional conditions, 
and so the entire Cartan connection 
is constructed from these simple data by the general theory.
Thus, our main
technical step will be to
observe that generic real submanifolds of dimension $2n+n^2$ in $\C^{n+n^2}$
inherit at each point 
such an infinitesimal structure from the ambient complex space.

The structural information on the parabolic geometries is encoded neatly in
cohomological terms. The curvature form $\Om\in \Om^2(\G,\g)$
of the Cartan connection $\om$ is given by the structure equation
$$
\Om = d\om +\frac12[\om,\om]
$$
and the absolute parallelism identifies the curvature with the
curvature function
$$
\kappa:\G\to \wedge^2 \pp_+\otimes\g,\qquad \kappa(X,Y) =
K(\om^{-1}(X),\om^{-1}(Y))
.$$
Thus, the curvature function
has values in the cochains of the Lie algebra cohomology of $\g_-$ with
coefficients in $\g$. This cohomology is explicitly computable by the
Kostant's version of the BBW theorem, cf. \cite{silhan, CS09}. We may
compute it either by means of the standard differential $\partial$ or by its
adjoint co-differential $\partial^*$. The formula in the special case of
the above two-chains is
$$
\partial^*(Z_0\wedge Z_1\otimes X) = -Z_0\otimes [Z_1,X] + Z_1\otimes[Z_0,X]
-[Z_0,Z_1]\otimes X
.$$
The normalization procedure relies on another important property of
the parabolic geometries, which imposes conditions
on the behaviour of the filtrations and is called {\em
regularity}. In words, the filtrations have to respect the Lie brackets of
vector fields and coincide with the Lie algebra structure $\gg_-$ at the
graded level. In terms of the curvature, this says that no curvature
components of non-positive homogeneities are allowed,
cf. \cite[Section 3.1]{CS09}.

The general theory shows that normalizing the regular Cartan connections by the
co-closedness $\partial^*\kappa=0$ of the curvature defines an equivalence of categories
of certain filtered manifolds (with additional simple geometric structures under some
cohomological conditions, like for all $|1|$-gradings or contact gradings) and
categories of Cartan connections, cf. \cite[Section 3.1]{CS09}. Then the harmonic part of
the curvature (which is a well defined quotient) defines all the rest and, in particular,
the geometry is locally isomorphic to its flat model if and only if the harmonic
curvature vanishes. Moreover, the entire curvature tensor is computable explicitly by a
natural differential operator from its harmonic part, cf. \cite{cap-twistor}.

\subsection{Free CR-distributions}

Let us come back to our main example, the structures modelled over the
parabolic homogeneous space $G/P$ with $G=\PSU(n+1,n)$, $n>1$, and $P$ as above.
In general, the first cohomology $H^1(\g_-, \g)$ governs the amount of data determining
the regular infinitesimal structures, cf. \cite[Section 4.3]{CS09}.
Indeed, the distribution itself
encodes the entire geometry if the first cohomology $H^1(\g_-, \g)$ appears
only in negative homogeneities. A direct computation checks that this
happens in our case.

Finally, we have to find out the conditions for the regularity, which is
more sophisticated.  At the first glance it seems that the complex structure
on $\g_{-1}$ and the choice of the reductive part $G_0$ of $P$ should be an
important part of the geometric data, but the above mentioned cohomological
computation says that this cannot be true.  The reason is given by the lemma
below which uses the terminology of a totally real skew-symmetric form. 
Recall that a real skew-symmetric bilinear form $\omega$ on a complex space
$V$ with complex structure $J$ is called totally real if
$\omega(JX,JY)=\omega(X,Y)$.

\begin{lem}\label{lem2}
The only complex structures $\tilde J$ on the real
vector space $\g_{-1}$ which make the Lie bracket $\Lambda ^2\g_{-1}\to
\g_{-2}$ into a totally real skew-symmetric form valued in skew
Hermitian matrices are $\tilde J=\pm J$,
where $J$ is the standard complex structure on $\g_{-1}$.
Moreover, all linear
homomorphisms $A\in GL(\g_{-1})$ allowing an extension $\tilde A$
to a Lie algebra
automorphism of $\g_-$ are
complex linear or complex anti-linear.
\end{lem}
\begin{proof}
Let us observe that the Lie bracket $\Lambda^2\g_{-1}\to\g_{-2}$ is actually
formed by $n^2$ ordinary skew-symmetric bilinear forms which are linearly
independent. Thus being totally real means that every skew-symmetric bilinear
form $\om\in\Lambda^2\g_{-1}^*$ in the span of the above ones 
has to be invariant with respect to the
endomorphism $J$. Now, considering both $\om$ and $J$ as matrices in a
chosen basis, this is equivalent to the property that the matrix composition
$J^T\om$ is symmetric for all skew-symmetric matrices $\om$ in question, 
thus $J\om$ is always symmetric, too.
  
Let $\tilde{J}$ be a real endomorphism of $\mathfrak g_{-1}$ with
$\tilde{J}^2=-\id$ and such that any $J$-invariant skew-symmetric bilinear
form $\omega$ is also $\tilde{J}$ invariant.  Thus, we request 
$\tilde{J}\omega$ symmetric for any skew-symmetric $\omega$ such
that $J\omega$ is symmetric.

Let $\C^{2n}$ be the complexification of $\R^{2n}=\C^n$ with coordinates
such that 
$$J=\begin{bmatrix}\I\ I&0\\0&-\I\ I \end{bmatrix}
.$$ 
Then $J\omega$ is symmetric for skew-symmetric $\omega$ if and only if
$$
\omega=\begin{bmatrix} 0 & B\\ -B^t &0 \end{bmatrix}$$ 
for some $n\times n$
matrix $B$.  

Now let 
$$
\tilde{J}=\begin{bmatrix}P&Q\\R&S \end{bmatrix}
$$ 
with
$\tilde{J}^2=- I$.  Then $\tilde{J}\omega$ is
$$
\begin{bmatrix}-QB^t&PB\\-SB^t&RB \end{bmatrix}
$$ 
which is symmetric for any
$B$ if and only if $Q=R=0$ and $P=-S=\lambda I$.  From
$\tilde{J}^2=- I$ we get $\lambda=\pm \I$ as required.

Finally, if $A$ in $\GL(\g_{-1})$ allows an extension into a morphism of the
Lie algebra $\g_-$, then
$$
[A^{-1}JAX, A^{-1}JAZ] =
\tilde A^{-1}\cdot [JAX,JAZ] = \tilde A^{-1}\cdot \tilde A[X,Z] = [X, Z]
$$
and so $A^{-1}JA = \pm J$ by the conclusion above.
\end{proof}

Now, we are led to the definition of the main objects of our
considerations:

\begin{df}
Consider a smooth manifold $M$ of real dimension $2n+n^2$ equipped with
a $2n$-dimensional distribution $D=T^{-1}M\subset TM$, such that
$[D,D]=TM$. We call $D$ a free CR distribution of dimension $n$
on $M$ if there is a fixed
almost complex structure $J$ on $D$ such that the algebraic
Lie bracket $\mathcal L:D\wedge D\to TM/D$ is totally real.
\end{df}

Let us remind, that if such a $J$ exists, then it is
unique up to the sign in view of Lemma 2. Thus we have to understand the
condition as a quite severe restriction on the distribution $D$ only. In
fact the existence of the second cohomology $H^2(\g_-, \g)$ of homogeneity
zero indicates, that actually we have to expect, that the generic distributions
of the proper dimension and codimension will not possess such a $J$. We
shall even see that if the dimension is bigger than two, then each such
$J$ has to be integrable (which can be understood as 
a differential consequence of our
requirement that the curvature vanishes in homogeneity zero). 

But clearly, the requirements of Definition 2 do not represent any problem 
for the
embedded CR-structures $M\subset \Bbb C^{n+n^2}$, since the bracket as well
as the complex structures are inherited from the ambient complex space.

\begin{lem}\label{lem3}
Every free CR distribution of dimension $n$ provides a regular
infinitesimal flag structure of type $(\PSU(n+1,n),P)$ on the
$(2n+n^2)$-dimensional manifold $M$.
\end{lem}
\begin{proof}
We have mentioned the properties of infinitessimal flag structures above 
and in the case
of free CR--geometries we just want to show, that each associated graded
tangent space, together with the Levi bracket $\mathcal
L$, is isomorphic to the the Lie algebra $\gg_{-2}\oplus \gg_{-1}$.
 
Let us consider a fixed tangent space $T_xM$ and the CR subspace
$D_x\subset T_xM$ and write $J$ for the complex structure on $D_x$. The
requirement on the Levi bracket $\mathcal L(JX,JY) = \mathcal L(X,Y)$ 
says that
$\mathcal L$ is the imaginary component of a (vector valued)
Hermitian form $\tilde {\mathcal L}:\Bbb C^n\otimes\Bbb C^n\to \Bbb C^{n^2}$
(the classical Levi form).
The non-degeneracy condition $[D,D]=TM$ ensures that this Hermitian form 
must be equivalent to the standard form mentioned in Lemma \ref{lem2}. Thus
identifying $\operatorname{Gr}(T_xM) = D_x\oplus (T_xM/D_x)$
with $\g_- = \g_{-1}\oplus\g_{-2}$
as graded Lie algebras is just what we need. 

The regularity follows by construction.
\end{proof}

\begin{thm}\label{thm1}
For each free CR distribution of dimension $n>1$ on a manifold $M$, there
is the unique regular normal Cartan connection
of type $(G,P)$ on $M$ (up to isomorphisms).

The only fundamental invariants of free CR distributions of dimensions $n>2$
are concentrated in the curvature of homogeneity degree $1$ and
correspond to the totally trace-free part of the $\sll(n,\C)$-submodule
$\Hom(\g_{-1}\otimes\g_{-2},\g_{-2})$ in the torsion.

In the case $n=2$, the same fundamental invariant exists and, additionally,
there is the Nijenhuis tensor $N(X,Y)=[X,Y]-[JX,JY]+J([JX,Y]+[X,JY])$ on the distribution, which vanishes automatically
on the embedded real eight-dimensional manifolds $M$ in $\Bbb C^6$.

Moreover, every smooth map between two free CR distributions respecting the
distributions is either a CR morphism or a conjugate CR morphism on the
connected components of $M$.
\end{thm}

\begin{proof}
The existence and uniqueness
of the Cartan connection follow from the general theory (see
\cite[Section 3.1]{CS09}). For
the explicit computation of the curvature see \cite[Section 4.3]{CS09}
or follow the algorithm of Kostant, see \cite{CS09, silhan}. In
particular, it turns out that there are two couples of complex conjugate
components for the complexified cohomologies. 

If $n>2$, one of them appears
in homogeneity zero and the corresponding part of the curvature is
excluded by the regularity condition already. The other one corresponds just
to the submodule as described in the Theorem. 

If $n=2$, both pairs of components appear
in homogeneity one and the additional one is generated (on the complexified
tangent space) by cochains mapping
two holomorphic elements to the anti-holomorphic image and vice versa. Thus
this part of the curvature expresses the Nijenhuis tensor
of the complex structure on $D$, see \cite[Section 4.2]{CS09} for an
analogous argumentation in the hypersurface CR case.

Finally, any smooth map respecting the distributions is a homomorphism of
the underlying infinitesimal structures, up to a possible change of the sign
of the chosen complex structures. Thus, such a map must be a morphism of the
unique normal Cartan connections on the connected components of $M$.
\end{proof}

\begin{cor}\label{cor1}
A generic real submanifold $M$ of real dimension $2n+n^2$ in $\Bbb
C^{n+n^2}$, $n>1$, inherits the free CR distribution 
structure from the ambient
complex structure. In particular, there is the canonical Cartan connection
for $M$, and $M$ is locally CR isomorphic to the homogeneous quadric 
$Q$ if and only if the
fundamental invariant torsion described in the Theorem vanishes.
\end{cor}

\begin{proof}
On a generic submanifold $M$ of the given dimension, the CR structure has
got CR dimension $n$ and its Levi form is totally real. Thus, the inherited complex
structure $J$ satisfies the assumptions on the free CR distributions and the
Corollary follows.
\end{proof}

\begin{cor}\label{cor2}
The group of all automorphisms of the homogeneous quadric $Q$ is
$\PSU(n+1,n)$.
\end{cor}

\begin{proof}
By Corollary \ref{cor1} there is the free CR distribution structure on $Q$
and the Maurer-Cartan form on $G=\PSU(n+1,n)$
is clearly the normal regular Cartan connection
for the geometry in question. Thus, the uniqueness result and the Liouville
theorem on automorphisms of homogeneous spaces (cf. \cite[Section 1.5]{CS09})
show that all local automorphisms of this CR free distribution are
restrictions of left shifts by elements in $G=\PSU(n+1,n)$.
\end{proof}

\begin{cor}\label{cor3} 
The complex structure $J$ on
a free $CR$ distribution of dimension $n>2$ is always integrable.  
\end{cor}

\begin{proof}
The general theory of parabolic geometries ensures that the homogeneity $1$
component of the torsion consists just of the harmonic components. 
Since the component of the torsion corresponding to the Nijenhuis tensor of
$J$ is not listed among the harmonic ones, and its homogeneity is one, this
part of torsion has to vanish.
\end{proof}

\section{The fundamental invariant of free CR distributions}

In order to exploit the existence of the unique normal Cartan connection
$\om$ for free CR distributions deduced in Theorem \ref{thm1}, 
we do not need to
compute the whole canonical frame in the classical way.
We shall rather make only the first step in the standard 
prolongation procedure, i.e.~we compute
the necessary objects of homogeneity one. This will provide us with the
fundamental invariants, as well as the splitting of the $P$ modules into 
the $G_0$-modules if their length of grading is at most two. In particular,
this provides the distinguished complementary subspaces 
to the CR distribution on the tangent
space. The latter data will be also 
sufficient for the Fefferman construction later on.

\subsection{Structure equations for free CR distributions}
For the sake of
simplicity, we shall restrict ourselves to the case of integrable complex
structure on $D$, which is always satisfied for embedded free CR
distributions and does not represent any restriction for ranks $n>2$.

In the Cartan-Tanaka procedure, all information related to homogeneity $1$
objects is obtained
after the first prolongation step (the bottom-to-top approach). The construction
in \cite[Section 3.1]{CS09} provides the entire Cartan connection and
the complete information on the structure of the curvature, without the
explicit prolongation. We shall combine
these two approaches by using the detailed knowledge on the curvature during
the explicit computation. In the sequel, we shall always
ignore the lowest dimensional case with $n=1$ since the curvature structure
is different and this case of real hypersurfaces in $\Bbb C^2$
is well known.

Let $D$ be a non-degenerate rank $2n$ vector distribution 
on a manifold $M$ of dimension
$2n+n^2$, $n>1$. We assume that a complex structure $J$ on $D$ is given such
that the Levi form $\mathcal L:\Lambda^2 D\to TM/D$
defined by the Lie bracket of vector fields is totally real. As usual, we
shall work in the complexification $TM\otimes \Bbb C$. The bundle $D\otimes \Bbb C$ splits into the holomorphic and
anti-holomorphic eigen-subbundles $D^{1,0}$ and $D^{0,1}$ with respect to the complex
structure $J$. Typical
sections of $D^{0,1}$ are written as $\xi+iJ\xi$ for $\xi\in D$.
The complex bilinear extension of $\mathcal L$ evaluates on such sections as
$$
\mathcal L(\xi+iJ\xi, \zeta+iJ\zeta) = (\mathcal L(\xi,\zeta) - \mathcal L(J\xi,
J\zeta)) + i(\mathcal L(J\xi,\zeta) + \mathcal L(\xi,J\zeta))
$$
and so the assumption that $\mathcal L$ is totally real is equivalent to the
vanishing of this expression. This is, in fact, equivalent to the fact that the (complexified)
Lie bracket of two vector fields in $D^{0,1}$ lies in $D\otimes \Bbb C$ and
similarly for $D^{1,0}$.

Now, let us fix a basis $X_1,\dots, X_n$ of $D^{1,0}$, and the complex
conjugate basis
$X_{\bar 1},\dots,X_{\bar n}$ of $D^{0,1}$. The non-degeneracy condition implies that the $n^2$
vector fields $X_{i\bar j} = - [X_i,X_{\bar j}]$, $\quad 1\le i,j \le n$,
complete them to a basis of the complexified
tangent space $TM\otimes \Bbb C$ at all points where the fields $X_i$,
$X_{\bar i}$ are defined. We shall consider a choice corresponding to
a complex basis $\xi_1,\dots,\xi_n$ of $D$ as above and the two obvious
complex conjugate bases of $D^{0,1}$ and $D^{1,0}$.
Moreover, since the almost complex structure on $D$ is integrable,
the distributions
$D^{0,1}$ and $D^{1,0}$ are involutive and we may assume that their generators
commute.

With these data at hand, we may start the computations of the fundamental
torsion invariant from Theorem \ref{thm1}.

Let $\{X_i,X_{\bar i},X_{i\bar j} \}$ be any local frame of $TM\otimes \Bbb
C$ as above.  To indicate the Hermitian skew symmetry
$X_{j\bar{i}}=-\overline{X_{i\bar{j}}}$ we write $X_{[i\bar{j}]}$.  Denote
by $\{ \th^i, \th^{\bar i}, \th^{[j\bar k]} \}$ the dual coframe of
$T^*M\otimes \Bbb C$. In fact, we shall be working on the real tangent space
$TM$ and so we shall need the coframe $\{ \th^i, \th^{[j\bar k]} \}$ where
the forms $\th^i$ represent a real form valued in $\Bbb C^n$, 
$\th^{\bar i} = \overline{\th^i}$, 
and the forms $\th^{j\bar k}$ represent a real form valued
in skew Hermitian matrices.

Let us write $D^{\perp}$ for the set of all 1-forms on $M$ annihilating
$D$.  It is clear that $D^{\perp}$ is generated by $\th^{[j\bar k]}$.

Note that $d\th^{[j\bar k]}(X_r, X_{\bar s}) = - \th^{[j\bar k]}([X_r,X_{\bar s}]) = \delta^j_r\delta^{\bar k}_{\bar s}$,
while $d\th^{[j\bar k]}(X_r, X_s) = 0$ because of the integrability. 
This implies that
$$
d\th^{[j\bar k]} = \th^j\wedge\th^{\bar k} \mod \langle \th^{[r\bar s]}\rangle.
$$
So, the structure equations of the coframe $\{\th^i,\th^{[j\bar k]}\}$ have
the form (here and below we use the Einstein summation convention):
\begin{equation}\label{streq}
\begin{aligned}
d\th^r =& f^r_{ij\bar k}\th^i\wedge\th^{[j\bar k]}+ f^r_{\bar{i}j \bar k}\th^{\bar i}\wedge\th^{[j\bar k]} + f^r_{i\bar{j}k\bar{l}}\th^{[i\bar{j}]}\wedge \th^{[k\bar{l}]}, \\
d\th^{[r\bar{s}]} =& \th^r\wedge \th^{\bar{s}} + f^{r\bar{s}}_{ij\bar{k}}\th^i\wedge\th^{[j\bar{k}]} + f^{r\bar{s}}_{\bar{i}j\bar{k}}\th^{\bar{i}}\wedge\th^{[j\bar{k}]} +f^{r\bar{s}}_{i\bar{j}k\bar{l}}\th^{[i\bar{j}]}\wedge \th^{[k\bar{l}]},
\end{aligned}
\end{equation}
where $f^r_{ij\bar{k}}$, $f^r_{\bar{i}j\bar{k}}$, $ f^r_{i\bar{j}k\bar{l}}$,
$f^{r\bar{s}}_{ij\bar{k}}$,  $f^{r\bar{s}}_{\bar{i} j\bar{k}}=\overline{f^{s\bar{r}}_{ik\bar{j}}} $,  $f^{r\bar{s}}_{i\bar{j}k\bar{l}}$ are the structure functions of the coframe $\{\th^i, \th^{\bar{i}}, \th^{[j\bar{k}]}\}$ on $M$, which are uniquely determined by the choice of the frames $X_i, X_{\bar i}$.
The absence of terms of the form $\th^i\wedge\th^j$, $\th^i\wedge\th^{\bar{j}}$ and $\th^{\bar{i}}\wedge\th^{\bar{j}}$ in the first formula follows from the integrability and a suitable choice of the basis.

The integrability also implies 
that the $f^{r\bar{s}}_{ij\bar{k}}$ are symmetric in $i,j$. Indeed, we have
\begin{align*}
 f^{r\bar{s}}_{ij\bar{k}}=&d\theta^{r\bar{s}}(X_i,[X_j,X_{\bar{k}}])\\
 =&-\th^{r\bar{s}}([X_i,[X_j,X_{\bar{k}}]])\\
=& -\th^{r\bar{s}}([[X_i,X_j],X_{\bar{k}}]]) -\th^{r\bar{s}}([X_j,[X_i,X_{\bar{k}}]])\\
=& -\th^{r\bar{s}}([X_j,[X_i,X_{\bar{k}}]])= f^{r\bar{s}}_{ji\bar{k}}.
 \end{align*}
 
Note that these coefficients do not form any
tensor, since their transformation rule under the change of the frame involves
derivatives. 

\subsection{The prolongation in homogeneity one}
The natural Cartan connection associated with the distribution $D$
will allow us to construct the coframes behaving much nicer and we shall
obtain the components of the curvature tensor at the same time.

Let $\pi\colon \G\to M$ be any principal $P$-bundle on $M$ and
$\om \colon T\G\otimes \Bbb C\to \g\otimes \Bbb C$
be a regular and normal Cartan connection of type $G/P$ (which always exists
by Theorem 1). For any
section $s\colon M\to\G$ we can write explicitly: 
\[
s^*\om = \begin{pmatrix} \om^i_j & \om_i & \om_{[i\bar{j}]} \\
 \om^j & -2i \Im \tr \omega^i_j &  -\om_{\bar{i}} \\
   \om^{[i\bar{j}]} &- \om^{\bar{j}} & -\om^{\bar j}_{\bar{i}}\end{pmatrix}
\]

where $\om^{[i\bar{j}]}$, $\om_{[i\bar{j}]}$, $\om^i_j$, $\om^i$, $\om^{\bar j}_{\bar{i}}=\overline{\om^i_j} $, $\om_j$ are 1-forms on
$M$.

We say that the Cartan connection $(\G,\om)$ is \emph{adapted to the
distribution $D$}, if $D=\langle \om^{[i\bar{j}]}\rangle^\perp$. It is easy
  to see that this definition does not depend on the choice of the
  section $s$.

Let $(G,\om)$ be an adapted Cartan connection. Then we have
\[
 D^\perp = \langle \om^{[i\bar{j}]} \rangle = \langle \th^{[i\bar{j}]} \rangle,
\]
where, as above, the forms $\th^{[i\bar{j}]}$ are defined by fixing a frame
$\{X_1,\dots,X_n\}$ on $D^{1,0}$.

We can always choose such section $s\colon M \to \G$ that
\begin{equation}\label{norm0}
\om^i = \th^i \mod D^{\perp}.
\end{equation}
This condition defines $s$ uniquely up to the transformations $s\to sg$, where
$g\colon M\to P_{+}$ is an arbitrary $P_{+}$-valued gauge transformation.

Consider the component $\Om^{[i\bar{j}]}$ of the curvature tensor:
\[
\Om^{[i\bar{j}]}=d\om^{[i\bar{j}]} - \om^i\wedge\om^{\bar{j}}+\om^i_k\wedge
\om^{[k\bar{j}]}-\om^{[i\bar{k}]}\wedge\om^{\bar j}_{\bar k}.
\]

Let us remind that $\om$ is regular and so only positive homogeneities may
appear in the curvature. The vanishing of homogeneity zero component 
immediately implies that
\[
d\om^{[i\bar{j}]} = \om^i\wedge\om^{\bar{j}} = \th^i\wedge\th^{\bar{j}} \mod D^{\perp},
\]
and, hence
\begin{equation}\label{norm01}
\om^{[i\bar{j}]}=\th^{[i\bar{j}]}\quad \mbox{for all $1\le i<j\le n$.}
\end{equation}

Next, let us compute the curvature coefficients of degree $1$ together with the
section normalizations of the form $s\mapsto \tilde s = sg$, where $g$ takes
values in $\exp{\g_1}$. Any such transformation leads to the
following transformation of the pull-back forms $\tilde s^*\om$:
\begin{align*}
\tom^{[i\bar{j}]} & = \om^{[i\bar{j}]},\\
\tom^{i} & = \om^i + p_{\bar{j}}\om^{[i\bar{j}]},\\
\tom^i_j &= \om^i_j- p_j \om^i - \frac{1}{2} p_j p_{\bar{k}} \om ^{[i\bar{k}]},
\end{align*} 
where the functions $p_{\bar j}$ define the mapping $dg\colon M\to \g_1$,
$p_{\bar j}= \overline{p_j}$.
Assume that 
\begin{align*}
\om^i &= \th^i + C^i_{j\bar{k}}\om^{[j\bar{k}]},\\
\om^i_j &= A^i_{kj}\om^k  + B^i_{\bar{k}j} \om^{\bar{k}} \mod D^{\perp}
\end{align*} 
for some functions $A^i_{kj}$, $B^i_{\bar{k}j}$, $C^i_{j\bar{k}}$. They are transformed by the gauge
transformation $g$ according to the following formula:
\begin{align*}
\tilde A^i_{kj} & = A^i_{kj} - \delta^i_kp_j;\\
\tilde B^i_{kj} & = B^i_{kj}\\
\tilde C^i_{j\bar{k}} & = C^i_{j\bar{k}} +\delta^i_{j} p_{\bar{k}}.
\end{align*}
We can use the functions $p_j$ to control the trace $A^i_{ij}$. 
We compute modulo $\Lambda^1(M)\wedge D^{\perp}$

\begin{align*}
\Omega^{i\bar{j}}\equiv & \th^i\wedge \th^{\bar{j}} + f^{i\bar{j}}_{rs\bar{k}} \th^r\wedge \th^{[s\bar{k}]} + 
\overline{f^{j\bar{i}}_{rk\bar{s}}} \th^{\bar{r}}\wedge \th^{[s\bar{k}]}- \om^i\wedge\om^{\bar{j}}+\om^i_k\wedge
\om^{[k\bar{j}]}-\om^{[i\bar{k}]}\wedge\om^{\bar j}_{\bar k}\\
\equiv& (\om^i - C^i_{r\bar{k}}\om^{[r\bar{k}]})\wedge (\om^{\bar j} - \overline{C^j_{k\bar{s}}} \overline{\om^{[k\bar{s}]}})
+ f^{i\bar{j}}_{rs\bar{k}} \om^r\wedge \om^{[s\bar{k}]} + 
\overline{f^{j\bar{i}}_{rk\bar{s}}} \om^{\bar{r}}\wedge \om^{[s\bar{k}]}\\& - \om^i\wedge\om^{\bar{j}}+A^i_{rk}\om^r\wedge
\om^{[k\bar{j}]}+B^i_{\bar{r}k}\om^{\bar r}\wedge
\om^{[k\bar{j}]}+ \overline{A^j_{rk}} \om^{\bar r} \wedge \om^{[i\bar{k}]}+ \overline{B^j_{\bar{r}k}} \om^{r} \wedge \om^{[i\bar{k}]}.\end{align*}
It follows
\[ \Om^{[i\bar{j}]} \equiv P^{i\bar{j}}_{rs\bar{t}}\om^r\wedge\om^{[s\bar{t}]}  +\overline{ P^{j\bar{i}}_{rt\bar{s}}}\om^{\bar r}\wedge\om^{[s\bar{t}]} \mod \wedge^2D^\perp,\]
where
\[P^{i\bar{j}}_{rs\bar{t}}= f^{i\bar{j}}_{rs\bar{t}} + A^i_{rs} \delta^{\bar{j}}_{\bar{t}}+ \overline{B^j_{\bar{r}t}} \delta^{i}_{s} +\overline{C^j_{t\bar{s}}}\delta^i_r. \]

Next we compute the torsion part of the curvature
\[ \Om^i = d \om^i - \om^j\wedge \om^i_j + (\om^j_j - \overline{\om^j_j})\wedge\om^i + \om_{\bar{j}}\wedge\om^{[i\bar{j}]}\]
modulo $D^\perp$. We have
\[ d\om^i\equiv d\th^i + C^i_{j\bar{k}} \th^j\wedge \th^{\bar{k}} \equiv C^i_{j\bar{k}} \om^j\wedge \om^{\bar{k}} \]
\begin{align*}
 \Om^i &= d \om^i - \om^j\wedge \om^i_j + (\om^j_j - \overline{\om^j_j})\wedge\om^i + \om_{\bar{j}}\wedge\om^{[i\bar{j}]}\\
 &\equiv   C^i_{j\bar{k}} \om^j\wedge \om^{\bar{k}} +   A^i_{kj} \om^k \wedge \om^j +   B^i_{\bar{k}j} \om^{\bar{k}} \wedge \om^j \\& - (A^j_{kj}\om^k +B^j_{\bar{k}j}\om^{\bar{k}}  - \overline{A^j_{kj}}\om^{\bar{k}}-\overline{B^j_{\bar{k}j}}\om^{k})\wedge \om^i.
\end{align*} 
Hence
\[ \Om^i= Q^i_{rs}\om^r\wedge\om^s  + Q^i_{r\bar{s}} \om^r\wedge\om^{\bar s}  \mod  D^{\perp} , \]
where
\begin{align*}
 Q^i_{[rs]}&=  A^i_{[rs]} - \frac12(A^j_{rj} \delta^i_{s}- A^j_{sj} \delta^i_{r}) + \frac12 (\overline{B^j_{\bar{r}j}} \delta^i_{s}- \overline{B^j_{\bar{s}j}} \delta^i_{r}) \\
 Q^i_{r\bar{s}}&= C^i_{r\bar{s}} -B^i_{\bar{s}r}- \overline{A^j_{sj}}\delta^i_r+ B^j_{\bar{s}j}\delta^i_r.
\end{align*}

Assuming that the connection does not have the torsion in the term
$\Hom(\wedge^2\g_{-1},\g_{-1}$) (as it is implied from the normality assumption
via Kostant's theorem in dimensions $n>2$ and by our integrability
assumption for $n=2$), we get $Q^i_{[rs]}=0$, $Q^i_{r\bar{s}}=0$, that is
\begin{align*}
 A^i_{[rs]} &=\frac12 (A^j_{rj} \delta^i_s - A^j_{sj} \delta^i_r)- \frac12 (\overline{B^j_{\bar{r}j}} \delta^i_{s}- \overline{B^j_{\bar{s}j}} \delta^i_{r}) ,\\
 C^i_{r\bar{s}} &= B^i_{\bar{s}r}+ \overline{A^j_{sj}} \delta^i_r - B^j_{\bar{s}j}\delta^i_r.
\end{align*}

It follows from the first equality that
\[ A^i_{is}= (2-n) A^j_{sj} + (n-1) \overline{B^j_{\bar{s}j}}. \]

We consider the cases $n>2$ and $n=2$ now separately. Thus let
us assume $n> 2$. 

If we make $A^i_{is}=0$, using the freedom in $\mathfrak g_1$, we have
\[A^j_{sj}= \frac{n-1}{n-2} \overline{B^j_{\bar{s}j}}\]
and
\[ C^i_{r\bar{s}} = B^i_{\bar{s}r}+ \frac{1}{n-2} B^j_{sj} \delta^i_r. \]

Substituting the second equality into the expression of $P^{i\bar{j}}_{rs\bar{t}}$ we get:
\begin{equation}\label{tensorP}
P^{i\bar{j}}_{rs\bar{t}} = f^{i\bar{j}}_{rs\bar{t}}
+ A^i_{rs} \delta^{\bar{j}}_{\bar{t}}+\overline{B^j_{\bar{r}t}} \delta^{i}_{s}+ \overline{B^j_{\bar{s}t}}\delta^i_r +\frac{1}{n-2}\overline{B^l_{\bar{s}l}}\delta^{\bar{j}}_{\bar{t}}\delta^i_r.
\end{equation}
Notice that
\[P^{i\bar{j}}_{[rs]\bar{t}} = f^{i\bar{j}}_{[rs]\bar{t}}=0\]
automatically.
Now,
\begin{equation}\label{tensorPsym}
P^{i\bar{j}}_{(rs)\bar{t}} = f^{i\bar{j}}_{rs\bar{t}}
+ A^i_{(rs)} \delta^{\bar{j}}_{\bar{t}}+2\overline{B^j_{(\bar{r}|t|}} \delta^{i}_{s)} +\frac{1}{n-2}\overline{B^l_{(\bar{s}|l}}\delta^{\bar{j}}_{\bar{t}|}\delta^i_{r)}.
\end{equation}
The traces of $P^{i\bar{j}}_{(rs)\bar{t}}$ are
\begin{align*}
 P^{i\bar{j}}_{(rs)\bar{j}} &= f^{i\bar{j}}_{rs\bar{j}}
+ n A^i_{(rs)} +2\overline{B^j_{(\bar{r}|j|}} \delta^{i}_{s)} +\frac{n}{n-2}\overline{B^l_{(\bar{s}|l|}}\delta^i_{r)} \\
&=f^{i\bar{j}}_{rs\bar{j}}
+ n A^i_{(rs)} +\frac{3n-4}{n-2}\overline{B^l_{(\bar{s}|l|}}\delta^i_{r)}\\
 P^{i\bar{j}}_{(is)\bar{j}} &= f^{i\bar{j}}_{is\bar{j}}
+ \frac{n}{2} A^i_{si} +\frac{(3n-4)(n+1)}{2(n-2)}\overline{B^l_{\bar{s}l}}=  f^{i\bar{j}}_{is\bar{j}}
 +\frac{(2n^2-n-2)}{n-2}\overline{B^l_{\bar{s}l}}\\
 P^{i\bar{j}}_{(is)\bar{t}} &= f^{i\bar{j}}_{is\bar{t}}
+\frac12 A^i_{si} \delta^{\bar{j}}_{\bar{t}}+(n+1)\overline{B^j_{\bar{s}t}}  +\frac{n+1}{n-2}\overline{B^l_{\bar{s}l}}\delta^{\bar{j}}_{\bar{t}}\\& = f^{i\bar{j}}_{is\bar{t}}
+(n+1)\overline{B^j_{\bar{s}t}}  +\frac{3n+1}{2(n-2)}\overline{B^l_{\bar{s}l}}\delta^{\bar{j}}_{\bar{t}}.
 \end{align*}
Since we have assumed that $n\ge 3$, we can uniquely determine coefficients
$A^i_{kj}$ and $B^i_{\bar{k}j}$ so that the tensor
$P^{i\bar{j}}_{rs\bar{t}}$ is totally trace-free. This in turn defines the
coefficients $C^i_{r\bar s}$ as well and we are done.

In the case $n=2$ the computations are a bit simpler. 
Going back to our general equalities and normalizing the section so that the
trace $A^i_{is}$ vanishes, we immediately conclude that the trace
$B^j_{\bar sj}$ vanishes too. Consequently,
\begin{align*}
 A^i_{[rs]} &=\frac12 (A^j_{rj} \delta^i_s - A^j_{sj} \delta^i_r),\\
 C^i_{r\bar{s}} &= B^i_{\bar{s}r}+ \overline{A^j_{sj}} \delta^i_r.
\end{align*}
Thus we can again substitute into the expression for $P^{i\bar j}_{rs\bar t}$
to obtain
$$
P^{i\bar{j}}_{rs\bar{t}} = f^{i\bar{j}}_{rs\bar{t}}
+ A^i_{rs} \delta^{\bar{j}}_{\bar{t}}+\overline{B^j_{\bar{r}t}} \delta^{i}_{s}
+ \overline{B^j_{\bar{s}t}}\delta^i_r + A^k_{sk}\delta^{\bar j}_{\bar
t}\delta^{i}_r.
$$
A direct check and computation of the traces again reveal that the
expression is automatically symmetric in the indices 
$r$ and $s$ and the $A$'s and $B$'s (and thus also $C$'s) are uniquely
determined by requiring the traces to vanish.

Summarizing, we have verified the following:
\begin{thm}\label{thm2}
The trace-free tensor $P^{i\bar j}_{rs\bar t}$ computed above
is the only fundamental invariant of the non-degenerate free CR
distribution $D$ of rank $n\ge 3$
on a manifold of dimension $2n+n^2$. Hence, the Cartan
connection associated with $D$ is locally 
flat if and only if this tensor vanishes identically.

In the case of non-degenerate real $4$-dimensional free CR-distributions
$D$ on $8$-dimensional manifolds, the Nijenhuis tensor of the complex
structure on $D$ together with the trace-free tensor 
$P^{i\bar j}_{rs\bar t}$ computed above form the complete system of
invariants. Hence, the Cartan
connection associated with $D$ is locally 
flat if and only if these tensors vanish identically.
\end{thm}

In particular, vanishing of 
the tensor $P^{i\bar j}_{rs\bar t}$ gives us the explicit condition when an
arbitrary non-degenerate free CR-distribution $D$ of rank $2n\ge4$ 
is locally equivalent
to the left-invariant distribution on the nilpotent Lie group corresponding
to the algebra~$\g_{-}$ (in the case of rank 4 we have to add the
integrability of the complex structure $J$ on $D$). Moreover, let us notice
that we obtain these complete data in the first prolongation step already.

\begin{rem}
The conclusion of Theorem 2 that the complex structures are always
integrable, except the lowest dimension of interest, is in fact 
not surprising. Just notice that the cohomology $H^2(\g_-,\g)$ lives in
homogeneities zero and one and we have excluded the zero homogeneity by the
regularity assumption on the infinitesimal flag structures. 
A generic normal Cartan connection of our type 
will have its Nijenhuis tensor of the
complex structure on $D$ as a differential consequence of the homogeneity
zero harmonic component (via the Bianchi identity and the BGG machinery).
Thus on embedded manifolds in $\Bbb C^{n+n^2}$, $n\ge2$, the
Nijenhuis tensor vanishes for common reasons. 
\end{rem}

\begin{rem}
Let us also comment on the geometric meaning of the
coefficients $A$, $B$ and $C$ which we computed during our prolongation step.

In general terms, these objects correspond to the choice of partial
Weyl connections (the coefficients $A^i_{jk}$ and $B^i_{\bar{k}j}$ define
the appropriate parts of the principal connection form, thus being the
usual Christoffel symbols of the corresponding affine connection) 
and the splittings of the
filtration (the improvement of the coframe by deforming $\th^i$ with the
help of the coefficients $C^i_{kl}$), cf.
\cite[Chapter 5]{CS09}. Both of
these objects have to be fixed together because they influence the same
curvature components in homogeneity one. This has been reflected by the
explicit link between $A$'s, $B$'s and $C$'s above. Of course, 
these partial Weyl connections
are analogs of the Webster-Tanaka connections for
hypersurface type CR geometries, cf. \cite[Section 5.2.12]{CS09}. The next
step of the prolongation procedure along the similar lines would complete
the connections to include also differentiation in the directions
complementary to $D$ and would compute the homogeneity two component of the
so called Rho tensor.
\end{rem}

\section{The Fefferman construction}

\subsection{The abstract functorial setting}\label{sec:fefferman}
The original Fefferman's construction of the circle bundle over a
hypersurface type CR manifold, equipped with a conformal structure, can be
presented in an abstract algebraic way, see \cite[Section 4.5]{CS09}.

Let $G/P$ and $\tG/\tP$ be two (real or complex) parabolic
homogeneous spaces and let us assume
\begin{itemize}
\item an infinitesimally injective homomorphism $i\colon G\to \tG$ is given,
\item the $G$-orbit of $o=e\tilde P\in\tilde G/\tilde
P$ is open (thus, $\frak g\to \tilde{\frak g}/\tilde{\frak
p}$ induced by $i':\frak
g\to\tilde{\frak g}$ is surjective),
\item $P\subset G$ contains $Q:=i^{-1}(\tilde P)$.
\end{itemize}

Consequently, there is the natural projection $\pi:G/Q\to G/P$,
$Q$ is a closed subgroup of $G$ (which is usually not parabolic). Moreover,
the homomorphism $i:G\to\tilde G$ induces the smooth map $G/Q\to \tilde
G/\tilde P$ which is a covering of the $G$-orbit of
$o$, and the latter open subset in $\tilde G/\tilde P$ carries
a canonical geometry of type $(\tilde G,\tilde P)$. This can be
pulled back to obtain such a geometry on $G/Q$.

In particular, it may happen that
$$
i(G)\tP=\tG\ \mbox{and}\
i(P)=i(G)\cap \tP
$$
i.e. $Q=P$ is the parabolic subgroup.
Then both parabolic geometries turn out to live over the same base manifold
$G/P = \tilde G/\tilde P$.
We say that $i$ is an
{\em inclusion of parabolic homogeneous spaces}.
This was the case of the free rank $\ell$ distributions with the spinorial
geometry on the Fefferman space, see \cite{DS09}.

The two steps discussed above are instances of two general 
functorial constructions on Cartan geometries $(\mathcal
G, \om)$
\begin{itemize}
\item the correspondence spaces
\item the structure group extensions.
\end{itemize}

The first one is given by a choice of subgroups $Q\subset P\subset G$
and it increases the underlying manifold $M=\mathcal G/P$
into a fibre bundle $\tilde M = \mathcal G/Q\to M$ with fiber $Q/P$.

The other one is based on embeddings of the structure group $G\to \tilde G$
and reasonable choices of subgroups $P\subset G$, $\tilde P\subset \tilde
G$,
and it leads to Cartan geometries on the same manifolds $M$, but with bigger
structure groups.

Combination of these two steps yields the Fefferman-like constructions as
shown above.

\subsection{$|1|$-graded parabolic geometry modelled on $\frak s\frak
u(n,n)$}
From the algebraic point of view, we have seen that the passage from free
distributions to free CR distributions consisted in replacing smartly
skew-symmetric matrices with the skew-Hermitian ones. 
Let us try to find the right way for generalizing the Fefferman construction
by inspecting carefully our lowest dimensional example with $n=1$, i.e. the
hypersurface type CR-structure embedded in $\C^2$. 

There, the original Fefferman construction provides a circle bundle $\tilde
M$ over $M$, equipped with a four-dimensional conformal structure in
Lorentzian signature. But the $4$-dimensional Lorentzian conformal geometry, in fact, can be equally well modelled as a $\mathfrak {su}(2,2)$ Cartan geometry.

Indeed, $\mathfrak{su}(2,2)$ can be represented by block matrices $\begin{pmatrix}A&Z\\X&-A^* \end{pmatrix}$ with $A\in \mathfrak{gl}(2,\mathbb C)$ such that $\operatorname{tr} A$ is real and $X,Z\in \mathfrak{u}(2)$. Then the components of grade $-1,0,1$ correspond to the lower-left, the diagonal and upper right blocks. The induced geometric structure is a choice of $\mathfrak{u}(2)$-frames on the 4-dimensional tangent spaces with structure group consisting of all regular complex $2\times 2$ matrices $A$ with real determinant acting on the skew-Hermitian frames $X$ by $(A^*)^{-1}X A^{-1}$. This action clearly preserves the conformal class of the real inner product induced on the tangent spaces by the determinant $\det X$.  (For details see \cite[4.1.10]{CS09}.)   

For a higher dimensional analogue of this geometry consider $\mathfrak{su}(n,n)$ with $|1|$-grading 
of $\mathfrak {su}(n,n)=\g_{-1}\oplus\g_{0}\oplus
\g_1$ as follows
\begin{align*}
\begin{matrix}\hspace{7mm} n \hspace{11mm}&  n \hspace{9mm}
\end{matrix}&\\
\begin{pmatrix}  \fbox{\rule[-3mm]{0cm}{10mm} \hspace{3mm} 0
\hspace{3mm}} &
\fbox{\rule[-3mm]{0cm}{10mm} \hspace{3mm} 1 \hspace{3mm}}\\
\rule{0cm}{.1mm}& \rule{0cm}{.1mm}\\
\fbox{\rule[-3mm]{0cm}{10mm} \hspace{2.5mm} -1 \hspace{2.5mm}} 
& \fbox{\rule[-3mm]{0cm}{10mm}
\hspace{3mm} 0 \hspace{3mm}}
\end{pmatrix}& \quad
\begin{matrix}\rule[-8mm]{0cm}{15mm}  n\\
\rule[-3mm]{0cm}{10mm}  n
\end{matrix}
\end{align*}

For a suitable choice of the pseudo-Hermitian form in the split signature $\frak g_{\pm1}$ are the spaces
of skew-Hermitian matrices with respect to the anti-diagonal, 
while $\frak g_0 = \frak s\frak l(n,\Bbb C)\oplus \Bbb R$ with 
pairs of matrices $(A, -\bar A^T)$ appearing always in the block diagonal, and
$\operatorname{tr}A\in\Bbb R$. Notice, $\g_{\pm1}$ are dual to each
other with  respect to the Killing form.

Now, as for any $|1|$-graded parabolic geometry, the geometric structure is given by
the appropriate reduction of the frame bundle to the structure group $G_0$:

\begin{df}
The $|1|$-graded geometry modelled on $\mathfrak{su}(n,n)$ 
is given by a smooth identification of the tangent spaces
$T_pM$ to a real $n^2$-dimensional manifold $M$ with the space of
skew-Hermitian $2$-forms $\Lambda^2_{\operatorname{skew-H}}(\mathcal S)$
on an auxiliary complex $n$-dimensional bundle $\mathcal S$ over $M$. Or, alternatively, it can be given by a smooth identification of the tangent spaces with $n\times n$ skew-Hermitian matrices $\psi_p\colon T_pM\to \mathfrak u(n)$, where two such identifications $\psi^1$ and $\psi^2$ are equivalent if $\psi_p^1= A^*(p) \psi_p^2 A(p)$ for some $\operatorname{GL}(n,\mathbb C)$-valued function $A(p)$.
\end{df}

We shall
write $\Bbb S\simeq \Bbb C^n$ for the standard fiber of the auxiliary bundle.
Notice,
that the purely imaginary part of the centre in $\GL(n,\Bbb C)$ acts on
$\Lambda^2_{\operatorname{skew-H}}(\Bbb C^n)$ trivially and so our 
$\g_0$ is the appropriate algebra to discuss here. 

Still, there is the $\Bbb
Z_2$ kernel of the action of the real multiples of identity matrices. But at the
group level, 
we shall work with $G=\SU(n,n)$ and so $G_0$ will be the group of all regular 
complex
matrices with real determinant. This means that we actually work with the
analogue to the choice of a spin structure on $M$ in the four-dimensional
case and our structure group $G_0$ is a double covering of the effective group
$G_0/\Bbb Z_2$ coming from the effective Klein geometry of the type in
question.

\subsection{The Fefferman construction for the homogeneous model}
Similar to the case of free distributions in \cite{DS09} and according to
the construction from Section \ref{sec:fefferman} we choose an embedding of
$\g = \frak s\frak u(n+1,n)$ to the $|1|$-graded Lie algebra $\tilde {\g} =
\mathfrak{su}(n+1,n+1)$ $$ \begin{pmatrix} A & X & Y \\
-{}Z^* & 2\alpha & -{}X^* \\
T & Z & -{}A^*
\end{pmatrix}
\mapsto 
\begin{pmatrix}
A & \frac{1}{\sqrt{2}} X & \frac{1}{\sqrt{2}} X & Y \\
- \frac{1}{\sqrt{2}} Z^* & \alpha & \alpha & - \frac{1}{\sqrt{2}} X^* \\
- \frac{1}{\sqrt{2}} Z^* & \alpha & \alpha & - \frac{1}{\sqrt{2}} X^* \\
T & \frac{1}{\sqrt{2}} Z & \frac{1}{\sqrt{2}} Z & -A^*
\end{pmatrix}
$$
where $A, Y, T\in\Mat_{\ell}(\C)$, $X,Z\in \C^\ell$, 
$Y+Y^*=T+T^*=0$, $\alpha = -i\operatorname{Im}\operatorname{Tr}A$.

Now, $\tilde {\frak p}$ is the parabolic subalgebra corresponding to the
only $|1|$-graded geometry as discussed above.  
Similar to the hypersurface CR case, the preimage $\frak q$ 
of $\tilde {\frak p}$ is
nearly the entire $\frak p$, with just one dimension in the centre $\Bbb C$
of $\frak g_0$ lacking.

The general functorial construction of the Fefferman space works in the
homogeneous model as follows.
We first consider the quotient of $G=\SU(n+1,n)$ by all of the $P$ but the
centre of $G_0$. This will provide a complex line bundle $\mathcal E(1,0)$
associated to the action of central elements $0\ne z\in \Bbb C$, $s\mapsto
z\cdot s$. Notice, we have adopted the same convention for weights as in the
hypersurface CR case, where $\mathcal E(a,b)$ stays for central action
$s\mapsto z^a\bar z^b\cdot s$.

Now, write $Q$ for the preimage of $\tilde P$ in
the embedding. Clearly, the 
requested space $G/Q$ is obtained by factorizing the action of the
real part of the centre on $\mathcal E(1,0)$ and 
thus $G/Q$ can be identified with the 
bundle of lines in $\mathcal E(1,0)$. 
This provides the circle bundle $\tilde M$, 
exactly as in the CR hypersurface case.

Moreover, the fixed grading of the Lie algebra $\g = \frak s\frak u(n+1,n)$
is well understood via the identification of $\g$ with the skew-Hermitian
matrices over the standard representation $\Bbb C^{n+2}$ and its splitting
into $\g_0$-submodules $\Bbb C^n \oplus \Bbb C \oplus \Bbb C^n$. This in
turn provides the identification of $\g/\frak q = \g_{-1}\oplus i\,\Bbb R$ with
skew-Hermitian forms on $\Bbb C^n\oplus \Bbb C$ and, thus, the requested
identification of the tangent bundle $T\tilde M$ with skew-Hermitian forms
on the auxiliary vector bundle with standard fibre $\Bbb C^{n+1}$.  

\subsection{The explicit construction in the curved case}
If we start with the normal Cartan connection $\om$ associated with a free
CR distribution on a manifold $M$ and 
the complex line bundle $\mathcal E(1,0)$ over $M$, then the categorial
construction provides the quotient manifold $\tilde M = \mathcal G/Q$, i.e. the
bundle of real lines in $\mathcal E(1,0)$, exactly as in the flat case. In
fact, we do not need the full Cartan bundle, since $\mathcal E(1,0)$ is a well
defined quotient bundle of the complex frame bundle $\mathcal G_0 = \mathcal
G/P_+$ of the
defining distribution $\mathcal D\subset TM$, at least locally. But
we have to be more careful with the requested $\frak s\frak u(n+1, n+1)$
geometry there.

Of course, we want to build the Fefferman circle bundle $\tilde M\to M$
including the required $|1|$-graded parabolic geometry
with the minimal effort straight from the free CR distribution itself.  
The key is the standard tractor calculus now.

Remember that the \emph{standard tractor bundle} is $\mathcal V M = \G
\times_P \Bbb V$, where $\Bbb V$ is the standard $SU(n+1,n)$ representation. 
The filtration $\Bbb V=\Bbb V^0\supset \Bbb V^1\supset \Bbb V^2\supset {0}$ 
induces the 
filtration of the tractor bundle 
$\mathcal V M =  \mathcal V^0M\supset \mathcal V^1M\supset \mathcal V^2M \supset 0$
on $\mathcal V M$ with codimensions $n+1$ and $n$. 

The standard representation carries a natural Hermitian form with respect
to which the adjoint representation is identified with the space of
skew-Hermitian maps on $\Bbb V$.  
As a $G_0$ module, the standard representation splits as 
$
\Bbb V = \Bbb C^n \oplus
\Bbb C
\oplus \Bbb C^n
$ 
and there is the $P$-module  
$
\Bbb S=\Bbb V/\Bbb V^2$, which decomposes as $\Bbb S = \Bbb C\oplus \Bbb C^n
$ and is dual to $\Bbb V^1$ as $G_0$-module.

Finally, the standard fiber of the tangent space to the circle bundle
$\tilde M$ is $\g_-\oplus i\,\Bbb R$ and it can be identified with 
a subspace of complex linear mappings $\Bbb V^1\to \Bbb S$. 
Exploiting the duality of $\Bbb V^1$ and $\Bbb S$ as $G_0$-modules, 
we may view them as skew-Hermitian 2-forms
on $\Bbb V^1$ as soon as we know the $G_0$-invariant splittings of $\Bbb
V^1$ and $\Bbb S$. 
By the very definition, this $G_0$-invariant identification 
is compatible with the inclusion $G\to \tilde G$.
Moreover, since the entire adjoint representation is identified as the space
of skew-Hermitian maps on $\Bbb V$, the definition of the $\frak s\frak u(n+1,n+1)$ geometry on
$\tilde M$ does not depend on the choice of the $G_0$ reduction. Thus, the identification is carried over to the level of the appropriate associated bundles and we may identify the tangent bundle to the
circle bundle $\tilde M$ with the skew-Hermitian second tensor product of
the auxiliary bundle $\mathcal S$, which is the definition of the
$|1|$-graded geometry introduced above.
Let us finally notice that the splitting of the auxiliary bundle 
$\mathcal S$ necessary for our identification is obtained
after the first prolongation already.

The necessary covering of the Cartan bundle for the free CR geometry 
by the $SU(n+1,n)$
geometry, ensuring the existence of the standard tractor
bundle $\mathcal VM$ and thus also $\mathcal S$, is 
clearly equivalent to the existence of the 
complex line bundle $\mathcal E(1,0)$. 

But for embedded free CR-manifolds this line bundle can be 
constructed similarly  to
the hypersurface case: There is the trivial canonical bundle $\mathcal K$
defined as the $(n+n^2)$-exterior power of the annihilator of the holomorphic
vectors in the complexified tangent bundle. Clearly the centre in $G_0$ acts
by the power $z^{-n-2n^2}$ on its standard fibre, 
and so we can take the appropriate root and consider
the dual space.

Having done this, we define the Fefferman space $\tilde M$ as the circle
bundle obtained by real projectivization of the complex line bundle $\mathcal
E(1,0)$. At the same time, the data from the first prolongation provide the
identification of the tangent bundle of $\tilde M$ with the bundle of
skew-Hermitian forms  $\Lambda^2_{\operatorname{skew-Herm}}\mathcal S$.
This concludes the geometric construction of the Fefferman
space, including its $|1|$-graded parabolic geometry structure.

\subsection{Concluding remarks and conjectures}
In general, the canonical normal Cartan connection $\tilde \om$ on 
the Fefferman
space $\tilde M$ 
does not need to be the one induced from the original connection $\om$ by 
the functorial construction. In such a case, they both define the same
underlying structure and so the difference is given by an adjoint tractor valued
1-form of positive homogeneity.

The functorial construction of the Cartan connection $\tilde \om$ links the
two curvatures $\kappa$ and  $\tilde\kappa$ in a simple algebraic way. In
principle, $\tilde \kappa = i'\circ \kappa$, where we view the curvatures as the
adjoint tractor valued two-forms on $\tilde M$ and $i'$ is induced
by the fixed Lie algebra
homomorphism $\frak s\frak u(n+1,n) \to \frak s\frak u(n+1,n+1)$. 

Moreover, we may use the same Lie algebra morphism $i'$ to rewrite this
formula at the level of cochains as a linear map 
$$
\phi : \Lambda^k\frak p_+\otimes \frak g \to \Lambda^k\tilde{\frak p}_+\otimes\tilde
{\frak g}
$$
and the question becomes {\it 
`under which conditions is the image $\phi(\kappa)$ a
$\tilde \partial^*$-closed element for a $\partial^*$-closed element
$\kappa$?'}.
  
These quite tedious computations were performed 
in the free rank $n$ distribution case in \cite{DS09} and they can be
performed with very minor extensions in our case. The conclusion reads:

\medskip
\noindent {\bf Claim.}
The Fefferman extension of a free CR-geometry to $|1|$-graded geometry on
the circle bundle $\tilde M$ is normal if and only if the $P$-invariant
restriction $\kappa_{1,1}$ of the entire curvature $\kappa$ to both
arguments in the complex distribution $\mathcal D\subset TM$ vanishes identically.

\smallskip
The recent results in the parabolic geometry, based on the detailed
understanding of the algebraic structure of the cohomology governing the
curvature, allow us to go much further (see \cite{CS09} for detailed
treatment). In particular, 
the general BGG machinery implies that the entire curvature is given by
a natural linear operator applied to the harmonic part of the curvature.
Since our projection to the component $\kappa_{1,1}$ is invariant too, every
projection of $\kappa_{1,1}$ to an irreducible component would be an invariant
operator with values in an invariant subbundle of two-cochains. But 
Kostant's version of the Bott-Borel-Weil theorem (cf. \cite{CS09}) 
guarantees that all the harmonic components appear in the entire 
space of cochains with multiplicity one. At the same time, there are only
linear operators mapping bundles coming from representations in 
second cohomologies into higher
cohomologies. Thus, such operators cannot exist as linear operators without
curvature in their symbols. 

Next, the Bianchi identity relates the differential and the 
fundamental derivative 
$$
\partial \kappa = \sum_{\operatorname{cycl}}i_{\kappa}\kappa - D\kappa 
$$ 
and employing $\partial^*\kappa=0$ we obtain for the lowest available
nontrivial homogeneity component $\square\kappa_{1,1}^0$ the appropriate
projection of
$$
\partial^* \sum_{\operatorname{cycl}}i_{\kappa}\kappa.
$$

Let us notice, that the homogeneity two obstruction is just a quadratic
tensorial expression. If this quantity vanishes, the next homogeneity
will become algebraic too, etc.
The detailed understanding of the normality obstruction in terms of the
harmonic curvature will require much more effort and it will be treated
elsewhere.

Let us 
conclude with an example of nontrivial geometry satisfying the above
normality condition for the Fefferman space. We adapt an example worked out
for the spinorial geometry by Stuart Armstrong, see \cite{Ar}. 

The $2n+n^2$-dimensional flat free CR manifold $Q$ can be described in coordinates $\{z_j, w_{kl}\}$ with $1\le j\le n$, $1\le k\le l\le n$, where $z_j, w_{kl} \in\mathbb C$ and $\Re w_{kk}=0$ for $1\le k \le n$ by the $D^{(1,0)}$ vector fields
$$Z_j= \frac{\partial}{\partial z_j} - \sum_{p=j}^n \bar{z}_p \frac{\partial}{\partial w_{jp}}.$$
Then
\begin{align*}
W_{kk}&=[Z_k,\bar{Z}_k]= \frac{\partial}{\partial w_{kk}} - \frac{\partial}{\partial \bar{w}_{kk}}&\\
W_{kl}&=[Z_k,\bar{Z}_l]= \frac{\partial}{\partial w_{kl}} & &\text{ if } k<l\\
W_{kl}&=[Z_k,\bar{Z}_l]= - \frac{\partial}{\partial \bar{w}_{lk}} &&\text{ if } k>l. 
\end{align*}

For $n\ge 4$ we modify $Q$ by replacing $Z_1$ by $$Z_1'=Z_1 +
\bar{w}_{12} \frac{\partial}{\partial w_{34}}.$$ Notice that $[Z_1',Z_j]=0$
for $2\le j\le n$, hence the modified CR structure is still integrable and
$[Z_1',\bar{Z}_j]= W_{1j}$.  The only resulting change in the structure
equations is that now $f^{[34]}_{1[12]}=1$.  It follows that the tensor $P$
is already trace-free for $A=B=0$, hence $A=B=C=0$ in this case and the only
non-vanishing coefficient in $P$ is $P^{[34]}_{1[12]}=1$.  Since the
curvature in homogeneity 1 is constant, by the Bianchi identity, the
curvature of higher homogeneity vanishes automatically.

\noindent
School of Science and Technology, 
University of New England, 
2351 Armidale, Australia
\\[2mm]
Department of Mathematics and Statistics,
Masaryk University, Kotl\'a\v rsk\'a 2,
611 37 Brno, Czech Republic

\end{document}